\documentclass[12pt]{amsart}

\usepackage{amsthm}
\usepackage{amsmath}
\usepackage{amssymb}
\usepackage{eucal}
\usepackage{amsfonts}
\usepackage{amsmath}
\usepackage[Lenny]{fncychap}
\usepackage{enumerate}
\usepackage{amssymb}
\usepackage[english]{babel}
\usepackage[latin1]{inputenc}
\usepackage[breaklinks=true]{hyperref}
\usepackage[all]{xy}
\usepackage{fancybox}
\usepackage{latexsym,mathrsfs,longtable,lscape,dsfont,yfonts}

\ifx\pdfpagewidth\undefined 
\usepackage[T1]{fontenc}
\else 
\usepackage[OT1]{fontenc} 
\fi 
\usepackage{amsfonts,amssymb,latexsym,longtable,amsmath}

\newtheorem{thm}{Theorem}[section]
\newcommand{\bthm}{\begin{thm}} \newcommand{\ethm}{\end{thm}}
\newtheorem{prop}[thm]{Proposition}
\newcommand{\bprp}{\begin{prop}} \newcommand{\eprp}{\end{prop}}
\newtheorem{fact}[thm]{Fact}
\newcommand{\bfct}{\begin{fact}} \newcommand{\efct}{\end{fact}}
\newtheorem{prob}[thm]{Problem}
\newcommand{\bprb}{\begin{prob}} \newcommand{\eprb}{\end{prob}}
\newtheorem{quest}[thm]{Question}
\newcommand{\bqtn}{\begin{quest}} \newcommand{\eqtn}{\end{quest}}
\newtheorem{lem}[thm]{Lemma}
\newcommand{\blem}{\begin{lem}} \newcommand{\elem}{\end{lem}}
\newtheorem{claim}[thm]{Claim}
\newcommand{\bclm}{\begin{claim}} \newcommand{\eclm}{\end{claim}}
\newtheorem{cor}[thm]{Corollary}
\newcommand{\bcor}{\begin{cor}} \newcommand{\ecor}{\end{cor}}
\newtheorem{conj}[thm]{Conjecture}
\newcommand{\bcnj}{\begin{conj}} \newcommand{\ecnj}{\end{conj}}
\theoremstyle{definition}
\newtheorem{defn}[thm]{Definition}
\newcommand{\bdfn}{\begin{defn}} \newcommand{\edfn}{\end{defn}}
\newtheorem{spec}[thm]{Specializing}
\newcommand{\bspc}{\begin{spec}} \newcommand{\espc}{\end{spec}}
\theoremstyle{remark}
\newtheorem{rem}[thm]{Remark}
\newcommand{\brem}{\begin{rem}} \newcommand{\erem}{\end{rem}}
\newtheorem{cnv}[thm]{Convention}
\newcommand{\bcnv}{\begin{cnv}} \newcommand{\ecnv}{\end{cnv}}
\newtheorem{exam}[thm]{Example}
\newcommand{\bexm}{\begin{exam}} \newcommand{\eexm}{\end{exam}}
\newcommand{\bpf}{\begin{proof}} \newcommand{\epf}{\end{proof}}

\newcommand{\sA} {{\mathcal A}}
\newcommand{\sB} {{\mathcal B}}

\newcommand{\sD} {{\mathcal D}}

\newcommand{\sG} {{\mathcal G}}

\newcommand{\sS} {{\mathcal S}}

\renewcommand{\phi}{\varphi}
\renewcommand{\theta}{\vartheta}

\newcommand{\A}{{\rm A}}

\newcommand{\ga}{{\alpha}}

\newcommand{\gb}{{\beta}}

\newcommand{\E}{{\rm E}}

\newcommand{\mkp}{\medskip}

\def\defi{\buildrel\rm def \over=}

\begin{document}

\title[Representation of group isomorphisms]{Representation of group isomorphisms. The compact case}

\author[M. Ferrer]{Marita Ferrer}
\address{Universitat Jaume I, Instituto de Matem\'aticas de Castell\'on,
Campus de Riu Sec, 12071 Castell\'{o}n, Spain.}
\email{mferrer@mat.uji.es}

\author[M. Gary]{Margarita Gary}
\address{Departamento de Matem\'aticas, Universidad Aut\'onoma Metropolitana, Iztapalapa, M\'exico DF, M\'exico}
\email{margaritagary1@hotmail.com}

\author[S. Hern\'andez]{Salvador Hern\'andez}
\address{Universitat Jaume I, Departamento de Matem\'{a}ticas,
Campus de Riu Sec, 12071 Castell\'{o}n, Spain.}
\email{hernande@mat.uji.es}

\thanks{ The first and third listed
authors acknowledge partial support by the Generalitat Valenciana,
grant code: PROMETEO/2014/062; and by Universitat Jaume I, grant P1·1B2012-05}

\begin{abstract}
Let $G$ be a discrete group and let  $\mathcal A$  and $\mathcal B$  be two subgroups of $G$-valued
continuous functions  defined on two $0$-dimensional compact spaces $X$ and $Y$. A group isomorphism
$H$ defined between $\sA$ and $\sB$ is called \emph{separating} when for each pair of maps
$f,g\in \sA$ satisfying that $f^{-1}(e_G)\cup g^{-1}(e_G)=X$, it holds that $Hf^{-1}(e_G)\cup Hg^{-1}(e_G)=Y$.
We prove that under some mild conditions every separating isomorphism
$H:\mathcal A\longrightarrow \mathcal B$ can be represented by means of a continuous function
$h: Y\longrightarrow X$ as a weighted composition operator. As a consequence we establish the equivalence of
two subgroups of continuous functions if there is a biseparating isomorphism defined between them.
\end{abstract}



\date{24 October 2014}

\maketitle \setlength{\baselineskip}{24pt}

\section {Introduction}

Let $G$ be a discrete group and let $X$ and $Y$ be topological spaces. If $\sA$ and $\sB$ are
groups of $G$-valued continuous maps, we say that $\sA$ and $\sB$ are \emph{equivalent}
when there is a homeomorphism $h\colon Y\to X$ and a continuous map $\omega\colon Y\to \A ut (G)$
satisfying that $Hf(y)=\omega[y](f(h(y)))$ for all $y\in Y$, where $\A ut(G)$ is equipped with the
pointwise convergence topology. We say in this case that $H$ is represented as a \emph{weighted composition
operator}.
There are many results that are concerned with the representation of linear operators as weighted composition
maps and the equivalence of specific groups of continuous functions in the literature, which is
vast in this regard. We will only mention here the classic Banach-Stone Theorem that, when $G$
is the field of real or complex numbers, establishes that if the Banach spaces of continuous
functions $C(X,G)$ and $C(Y,G)$ are isometric, then they are equivalent and the isometry can be
represented as a weighted composition map
(cf. \cite{Ara_Jar:2003,Banach,font-her,font-her2,gau-jeang-wong,gonzalez-uspenskij,SHdz:Houston,her-bec-nar,jarosz,Stone}).
Another important example appears in coding theory, where the well known MacWilliams Equivalence
Theorem asserts that, when $G$ is a finite field and $X$ and $Y$ are finite sets, two codes
(linear subspaces) $\sA$ and $\sB$ of $G^X$ and $G^Y$, respectively, are equivalent when they are
isometric for the
Hamming metric (see \cite{McW:i,McW:ii,Din_LP:2004}).
This result has been generalized to convolutional codes in \cite{GluLue:2009} and it also
makes sense in other areas, as for example functional analysis and linear dynamical systems
(cf \cite{Fer_Her_Sha:Products,forney_trott:04,GluLue:2009,rosenthal,willems:86}).
The main motivation of this research has been to extend MacWilliams Equivalence Theorem to more general
settings and explore the possible application of these methods to the study of convolutional codes
or linear dynamical systems. However, throughout this paper, we will only deal with $0$-dimensional
compact spaces $X$ and $Y$, and a discrete group $G$. We will look at the possible application of this
abstract approach elsewhere. There are many precedents in the study the representation of group
homomorphisms for group-valued continuous functions. Among them, the following ones are relevant here
(cf \cite{eda kiyosa ohta,Fer_Her_Rod,her_rod:2007,martinez,ohta 96,tesis-rodenas,yang-1,yang}).
Most basic facts and notions related to topological properties may be
found in \cite{engel}.

\mkp

Throughout this paper all spaces are assumed to be Hausdorff 0-dimensional and compact.
If $X$ is a topological space and $G$ is a topological (discrete) group, we denote by
$C(X,G)$ the group of continuous functions from $X$ to $G$. Let $e_G$ be the neutral element of $G$.
For $f\in C(X,G)$ the \emph{cozero} of $f$ is the set $coz(f)=\{x\in X : f(x)\neq e_G\}$ and
the \emph{zero} of $f$ is the set $Z(f)=X\setminus coz(f)$. Since $G$ is discrete $coz(f)$ and
$Z(f)$ are both closed and open (\emph{clopen}) subsets of $X$.
\mkp

Let $\sA$ be a subgroup of $C(X,G)$ and set $Z(\sA)\defi \{Z(f) : f\in\sA\}$. Then
$\sigma(Z(\sA))$ denotes the minimum collection of subsets containing $Z(\sA)$
that is closed under finite unions and intersections (resp. $coz(\sA)\defi \{coz(f) : f\in\sA\}$
and $\sigma(coz(\sA))$ denotes the minimum collection of subsets containing $coz(\sA)$ that is closed
under finite unions and intersections). It is said that $\sA$ \textit{separates points in} $X$ if
for every pair  $(x_1,x_2)\in X\times X$ there is a map $f\in\sA$ such that $f(x_1)\not=e_G$ and $f(x_2)=e_G$.
It is said that $\sA$ \emph{strongly separates points in} $X$ if
for every pair  $(x_1,x_2)\in X\times X$, there are maps $f_1,\, f_2\in\sA$ such that
$x_i\in coz(f_i), 1\leq i\leq 2$, and $coz(f_1)\cap coz(f_2)=\emptyset$.
\mkp

Denote by $\delta_x\colon\sA\to G$ the evaluation map,
that is, $\delta_x(f)=f(x)$ for every $f\in\sA$. It is said that $\sA$ is \emph{pointwise dense} when
$\delta_x(\sA)$ is dense in $G$ for all $x\in X$. 
It is said that
$\sA\subseteq C(X,G)$ is \textit{controllable}  if
for every $f\in\mathcal{A}$ and $D_1,D_2\in\sigma(Z(\sA))$ such that $D_1\cap D_2=\emptyset$ there exist a subset
$U\subseteq \sigma(coz(\sA))$ and a function $g\in \mathcal{A}$ such that $D_1\subseteq U\subseteq X\setminus D_2$,
$g|_{D_1}=f|_{D_1}$ and $g|_{Z(f)\cup (X\setminus U)}\equiv e_G$.
\mkp

We now formulate our main results.

\bthm
Let $X$ and $Y$ be $0$-dimensional compact Hausdorff spaces and let $G$ be a discrete group.
Suppose that $\sA$ and $\sB$ are controllable and pointwise dense subgroups of $G$-valued
continuous functions separating
the points of $X$ and $Y$, respectively.
If $H\colon\sA \to\sB$ is a biseparating group isomorphism of $\sA$ onto $\sB$,
then there are continuous maps
$h\colon Y\to X$ and $\omega\colon Y\to \A ut (G)$ satisfying the following properties:
\begin{enumerate}
\item $h$ is a homeomorphism of\, $Y$ onto $X$.
\item For each $y\in Y$ and every $f\in \sA$ it holds $$Hf(y)=\omega[y](f(h(y))).$$
\item $H$ is an isomorphism with respect to the pointwise convergence topology.
\item  $H$ is an isomorphism with respect to the compact open topology.
\end{enumerate}
\ethm
\mkp

\bcor
Let $X$ and $Y$ be $0$-dimensional compact Hausdorff spaces and let $G$ be a discrete group.
Suppose that $\sA$ and $\sB$ are controllable and pointwise dense subgroups of $G$-valued
continuous functions separating
the points of $X$ and $Y$, respectively.
If there is a biseparating group isomorphism $H$ of $\sA$ onto $\sB$,
then $\sA$ and $\sB$ are equivalent.
\ecor
\mkp

We notice that some of the requirements we have imposed on the previous results could be relaxed in general.
However, this would take us to a wider setting in general. For instance, if we assume that $\sA$ does not
separate points in $X$, then there must be some point $x\in X$
such that $f(x)=e_G$ for all $f\in\sA$, then we can replace $X$ by the largest subspace $X_\sA\subseteq X$
where $\sA$ separates points. This subspace $X_\sA$ is open but not necessarily closed in general.
Thus, the study of subgroups that does not separate points lead us to consider locally compact spaces. We shall
discuss these spaces in a subsequent paper.

\section{Basic notions and facts}

The following lemma is easily verified using a standard compactness argument.
Recall that we are assuming that all spaces are compact and $0$-dimensional.

\begin{lem}\label{lema00} Let $\sD$ be a family of clopen subsets of $X$ that is a subbase for the closed subsets of $X$.
Then for every disjoint nonempty closed subsets $A$ and $B$ of $X$ there are two disjoint subsets $D_A$ and $D_B$ in $\sigma(\sD)$
such that $A\subseteq D_A$ and $B\subseteq D_B$.
\end{lem}
\mkp

\bexm
A specific example where Lemma \ref{lema00} applies is given when $\sA$ separates points in $X$, where $\sD=Z(\sA)$.
\eexm
\mkp

Next proposition shows that the notions of separating and strongly separating points are equivalent for controllable subgroups.

\begin{prop}\label{prp01} If $\sA$ is a controllable subgroup of $C(X,G)$ that separates the points of $X$,
then $\sA$ strongly separates the points of $X$.
\end{prop}
\begin{proof}
 Set $\sD=Z(\sA)$ and take two distinct elements $x_1\neq x_2$ in $X$. Applying Lemma \ref{lema00},
since  $\sA$ separates the points of $X$, there are $D_1, D_2\in\sigma(\sD)$ such that $x_1\in D_1$, $x_2\in D_2$ and $D_1\cap D_2=\emptyset$.
Take $f_i\in\sA$ such that $f_i(x_i)\neq e_G$, $i\in\{1,2\}$. Since $\sA$ is controllable,
we have $U_1\in \sigma(coz(\sA))$ and $g_1\in\sA$ such that $D_1\subseteq U_1\subseteq X\setminus D_{2}$, $g_1|_{D_1}=f_1|_{D_1}$ and
$Z(f_1)\cup (X\setminus U_1)\subseteq Z(g_1)$. Therefore $x_1\in coz(g_1)\subseteq U_1$.

Applying that $\sD$ is a subbase of closed subsets, and using a compactness argument, we deduce
that there is $D'_1\in \sigma(\sD)$ such that $U_1\subseteq D'_1$ and $D'_1\cap D_2=\emptyset$. By the controllability of
$\sA$ again,
we have $U_2\in \sigma(coz(\sA))$ and $g_2\in\sA$ such that $D_2\subseteq U_2\subseteq X\setminus D'_{1}$, $g_2|_{D_2}=f_2|_{D_2}$ and
$Z(f_2)\cup (X\setminus U_2)\subseteq Z(g_2)$. Therefore $x_2\in coz(g_2)\subseteq U_2$, which yields
$coz(g_1)\cap coz(g_2)=\emptyset$. This completes the proof.
\end{proof}
\mkp


\bdfn Let $\sA$ be a subgroup of $C(X,G)$ and
let $\phi\colon\sA\to G$ a group homomorphism. A subset $A$ of $X$ is said to be a
\textit{support} for $\phi$ if given $f\in\sA$ with $A\subseteq Z(f)$, it holds that  $\phi(f)=e_G$.
\edfn

Some basic properties of support subsets are shown in the next proposition.
Observe that, since $A\subseteq \overline{A}^X\subseteq Z(f)$,
we may assume WLOG that all support subsets are closed and therefore compact subsets of $X$.
\mkp

\begin{prop}\label{propiedadessoporte}
Let $\phi\colon\sA\to G$ a non null group homomorphism. The following assertions hold :
\begin{enumerate}
\item $X$ is a support for $\phi$.
\item If $A$ is a support for $\phi$ then $A\neq\emptyset$.
\item If $A$ is a support for $\phi$ and $A\subseteq B$ then $B$ is a support for $\phi$.
\item Let $A$ be a support for $\phi$ and $f,g\in\sA$ such that $f|_A=g|_A$. Then $\phi(f)=\phi(g)$.
\end{enumerate}
\mkp

If, in addition, $\sA$ is controllable and separates points in $X$, then we have:

\begin{enumerate}
\item[(5)] Let $A$ and $B$ be supports for $\phi$, then $A\cap B\neq\emptyset$.
\end{enumerate}
\end{prop}
\begin{proof}
\noindent Assertions $(1) - (4)$ are obvious.

\noindent (5) Let $A$ and $B$ be closed supports for $\phi$. Suppose $A\cap B=\emptyset$. Since $\sA$ separates  points in $X$,
by Lemma \ref{lema00} and Proposition \ref{prp01}, there are two disjoint subsets $D_A$ and $D_B$ in $\sigma(Z(\sA))$ containing $A$ and $B$, respectively.
Take $f\in\sA$ such that $\phi(f)\neq e_G$. Applying the controllability of $\sA$, we obtain
$U\in \sigma(coz(\sA))$ and $g\in\sA$ such that $A\subseteq D_A\subseteq U\subseteq X\setminus D_B\subseteq X\setminus B$,
$g|_{D_A}=f|_{D_A}$ and $Z(f)\cup(X\setminus U)\subseteq Z(g)$. This yields a contradiction as the evaluation of $\phi(g)$ shows.
Indeed, since  $g(x)=f(x)$ for all $x\in A$, by item (2) it follows that $\phi(g)=\phi(f)\neq e_G$. On the other hand, we have that
$g(x)=e_G$ for all $x\in B$, which imples $\phi(g)=e_G$. This contradiction completes the proof.
\end{proof}


\bdfn
A map $H\colon\sA\to\sB$ is said to be \emph{separating, or disjointness preserving},
if for each pair of maps $f,g\in \sA$
satisfying that $f^{-1}(e_G)\cup g^{-1}(e_G)=X$,
it holds that $Hf^{-1}(e_G)\cup Hg^{-1}(e_G)=Y$
(equivalently, if $coz(f)\cap coz(g)=\emptyset$ implies $coz(Hf)\cap coz(Hg)=\emptyset$ for all $f,g\in\sA$).
In case $H$ is bijective, the map $H$ is said to be \emph{biseparating} if both $H$ and $H^{-1}$ are separating.
Remark that this definition makes sense and extends naturally to maps $\phi\colon\sA\to G$.
\edfn
\mkp

Next we will see that every non null separating group homomorphism $\phi\colon\sA\to G$, where
 $\sA$ is controllable, has the smallest possible compact support set.
  For that purpose, set $$\sS=\{A\subseteq X: A\text{ is a compact support for }\phi \}.$$
 There is a canonical partial order that can be defined on $\sS$: $A\leq B$, $A,B\in \sS$, if and only if
 $B\subseteq A$. A standard argument shows that $(\sS,\leq)$ is an inductive set and,
 by Zorn's lemma, $\sS$ contains a $\subseteq$-minimal element. Furthermore, this minimal element is in fact
 a minimum because of the next proposition. 

\begin{prop}\label{1elemento}
The minimum element in $\sS$ consists of a single element.
\end{prop}

\begin{proof}Let $S$ be a minimal element of $\sS$, which is nonempty by Proposition \ref{propiedadessoporte}.
Suppose now that there are two different elements $x_1,x_2$ that are contained in $S$. As $X$ is Hausdorff,
we can select two disjoint open subsets $V_1,V_2$ in $X$ such that $x_1\in V_1$ and  $x_2\in V_2$.
Since $S$ is minimal, the compact subset $S\setminus V_i$ is not a support for $\phi$. Hence,  there are
$f_i\in\sA$ such  that $ S\setminus V_i\subseteq Z(f_i)$ and $\phi(f_i)\neq e_G$, $1\leq i\leq 2$.
Since $\phi$ is separating, it follows that $A=coz(f_1)\cap coz(f_2)$ is a nonempty compact subset of $X$.

We claim that $S\cap A=\emptyset$. Otherwise, pick up an element $a\in S\cap A$.
If $a\in V_1$ then $a\in S\setminus V_2$ and $a\in Z(f_2)$, which is a contradiction;
but if $a\notin V_1$ then $a\in S\setminus V_1$, which implies that $a\in Z(f_1)$ and we get a contradiction again.
Therefore $S\cap A=\emptyset$. By Lemma \ref{lema00}, we can take two disjoint sets $D_S,D_A\in\sigma(Z(\sA))$
such that $S\subseteq D_S$ and $A\subseteq D_A$. Applying that $\sA$ is controllable to $D_S$, $D_A$ and $f_1$,
we obtain a  set $U\in\sigma(coz(\sA))$ and a map $g\in\sA$ such that
$S\subseteq D_S\subseteq U\subseteq X\setminus D_A\subseteq X\setminus A$, $g|_S=f_1|_S$ and
$g|_{Z(f_1)\cup(X\setminus U)}\equiv e_G$ . Then $U\cap A=\emptyset$, $\phi(g)=\phi(f_1)\neq e_G$ and $A\subseteq Z(g)$.
Since $\phi$ is separating the set $B=coz (g)\cap coz(f_2)\neq\emptyset$. Take $b\in B$. Then $b\in coz(f_2)$ and
$b\in coz(g)\subseteq coz(f_1)$, that is, $b\in coz(f_1)\cap coz(f_2)=A$. As a consequence
$g(b)=e_G$, which is a contradiction.  Therefore we have proved that $|S|=1$. This completes the proof.

\end{proof}

\section{Proof of main results}

Along this section $\sA$  (resp. $\sB$) is a controllable subgroup of $C(X,G)$
(resp. $C(Y,G)$) that separates points in $X$ (resp. $Y$).

Let $H\colon \sA\to\sB$ be a separating group homomorphism.
The maps $\delta_y\circ H$ are a separating group homomorphisms of $\sA$ into $G$ for all $y\in Y$.
Furthermore, since $\sA$ is controllable and separates points in $X$, we can apply Proposition \ref{1elemento},
in order to obtain that each partial ordered
set $\sS_y=\{A\subseteq X: S\text{ is a compact support for }\delta_y\circ H \}$ has a minimum element,
which is a singleton denoted by $h(y)$.
Therefore, by sending $y\in Y$ to $h(y)\in X$ for every $y\in Y$, we have defined the \textit{support map}
of $Y$ into $X$ that is associated to $H$.

\begin{prop}\label{propiedadesh}
Let $H\colon \sA\to\sB$ be a separating group homomorphism.
Then the support map $h$ has the following properties:
\begin{enumerate}
\item [1)] $h$ is continuous.
\item [2)] If $\emptyset\not=A\subsetneq X$ is open, $f\in\sA$ and $A\subseteq Z(f)$ then $h^{-1}(A)\subseteq Z(Hf)$.
\item [3)]  $h(coz(Hf))\subseteq coz(f)$.
\item [4)] If $H$ is one-to-one, then $h$ is onto.
\end{enumerate}
Moreover, when $H$ is a bijection of $\sA$ onto $\sB$, we have in addition:
\begin{enumerate}
\item [5)] If $H$ is biseparating,  then $h$ is a homeomorphism of\, $Y$ onto $X$.
\end{enumerate}
\end{prop}
\begin{proof}

1) Let  $(y_d)_{d\in D}$ be a net in $Y$ converging to $y\in Y$. By a standard compactness argument, we may assume WLOG that
$(h(y_d))_{d}$ converges to $x\in X$.
Reasoning by contradiction, suppose $h(y)\neq x$. Since $X$ is Hausdorff, we can take two disjoint open
neighborhoods $V_{h(y)}$ and $V_x$ of $h(y)$ and $x$, respectively. Using convergence, there is $d_1\in D$
such that $h(y_d)\in V_x $ for all $d\geq d_1$.

As every support subset for $\delta_{y'}\circ H$ contains $h(y')$, for all $y'\in Y$, the subset
$X\setminus V_{h(y)}$ may not be a support for $\delta_y\circ H$. Therefore there exists $f\in\sA$ such that
$X\setminus V_{h(y)}\subseteq Z(f)$ and $Hf(y)\neq e_G$. Moreover, since $H(f)$ is a continuous function, the net
$(Hf(y_d))_d$ converges to $Hf(y)$ and, because $G$ is discrete, there is $d_2\geq d_1$ such that $Hf(y_d)\neq e_G$ for all $d\geq d_2$.
If we take and index $d_3\in D$ such that $d_3\geq d_2$, then the subset $X\setminus V_x$ may not be a support for
$\delta_{y_{d_3}}\circ H$. Thus, there exists $f_3\in\sA$ such that $X\setminus V_{x}\subseteq Z(f_3)$ and $Hf_3(y_{d_3})\neq e_G$.
This means that $y_{d_3}\in coz(Hf_3)\cap coz(Hf)$ and, since $H$ is a separating map, it follows that
$coz(f_3)\cap coz(f)\neq\emptyset$. But $coz(f_3)\subseteq V_x$, which is disjoint from  $coz(f)\subseteq V_{h(y)}$.
This is a contradiction that completes the proof.

\noindent
 2) Let $\emptyset\not=A\subsetneq X$ be an open subset, $f\in\sA$ and $A\subseteq Z(f)$.
If we take $y\in h^{-1}(A)$, then $X\setminus A$ is a nonempty compact subset that is not a support for $\delta_y\circ H$.
Then there is $g\in\sA$ such that $X\setminus A\subseteq Z(g)$ and $Hg(y)\neq e_G$. Therefore $coz(g)\subseteq A$ and
$coz(f)\subseteq X\setminus A$. Since $H$ is separating, we have that $coz(Hg)\cap coz(Hf)=\emptyset$.
Therefore $Hf(y)=e_G$.

\noindent
3) Take $x\in h(coz(Hf))$. Then $x=h(y)$ for some $y\in coz(Hf)$. Since $\{h(y)\}$ is a support for $\delta_y\circ H$,
it follows that $x\notin Z(f)$ or, equivalently, we have $x\in coz(f)$.

\noindent
4) Suppose $h(Y)\neq X$ and take $x\in X$ such that $x\notin h(Y)$. Since $h$ is continuous and $Y$ is compact, we have that
$h(Y)$ is a proper compact subset of $X$. Applying Lemma \ref{lema00}, there are two disjoint subsets $D_x, D_Y\in\sigma(Z(\sA))$
such that $x\in D_x$ and $h(Y)\subseteq D_Y$. Moreover, as $\sA$ separates points in $X$, there exists $f\in\sA$ such that
$f(x)\neq e_G$. Again, by the controllability of $\sA$, we may take a subset $U\subseteq \sigma(coz(\sA))$ and a map
$g\in \sA$ such that $x\in D_x\subseteq U\subseteq X\setminus D_Y\subseteq X\setminus h(Y)$, $g|_{D_x}=f|_{D_x}$ and
$Z(f)\cup(X\setminus U)\subseteq Z(g)$. As a consequence $g(x)=f(x)\neq e_G$, $h(Y)\subseteq Z(g)$ and $H(f)(y)=e_G$ for all $y\in Y$.
Then $Hf\equiv e_G$. Since $H$ is an injective group homomorphism, this yields $f\equiv e_G$, which is a contradiction.

\noindent
5) Since $X$ and $Y$ are compact spaces, it will suffice to prove that $h$ is one-to-one.
Suppose there are two elements $y_1\neq y_2$ in $Y$ such that $h(y_1)=h(y_2)=x_0$.
Since $\sB$ separates the points of $Y$, there are $D'_1,D'_2\in \sigma(Z(\sB))$ such that
$y_1\in D'_1$, $y_2\in D'_2$  and $D'_1\cap D'_2=\emptyset$.
Since $H(\sA)=\sB$, there is $f_i\in\sA$ such that $H(f_i)(y_i)\neq e_G$, for $i\in\{1,2\}$.
Since $\sB$ is controllable, there are $U_1\in\sigma(coz(\sB))$, $g_1\in\sA$
such that $y_1\in D'_1\subseteq U_1\subseteq Y\setminus D'_2$, $Hg_1|_{D'_1}=Hf_1|_{D'_1}$ and
$Hg_1|_{Z(Hf_1)\cup(Y\setminus U_1)}\equiv e_G$. Since $D'_2\cap U_1=\emptyset$, applying a compactness
argument, there is $D''_1\in\sigma(Z(\sA))$ such that $U_1\subseteq D''_1$ and $D'_2\cap D''_1=\emptyset$.
Now, by the controllability of $\sB$, there are
$U_2\in\sigma(coz(\sB))$ and $g_2\in\sA$ such that $y_2\in D'_2\subseteq U_2\subseteq Y\setminus D''_1$,
$Hg_2|_{D'_2}=Hf_2|_{D'_2}$ and $Hg_2|_{Z(Hf_2)\cup(Y\setminus U_2)}\equiv e_G$. Hence, since $coz(Hg_i)\subseteq U_i$,
$U_1\cap U_2=\emptyset$, and $H$ is biseparating, it follows that $coz(g_1)\cap coz(g_2)=\emptyset$.
On the other hand $Hg_i(y_i)=Hf_i(y_i)\neq e_G$, $i\in\{1,2\}$, and by item 3) above, we have that
$h(y_1)=h(y_2)=x_0\in coz(g_1)\cap coz(g_2)$, which is a contradiction.
\end{proof}
\mkp

We have just seen how a separating group homomorphism $H$ has associated a continuous map $h$ that assigns to each point
$y\in Y$  the support subset of $\delta_y\circ H$. Our next goal now is to obtain a complete representation of $H$
by means of the support map $h$. Having this in mind, set $$G_{h(y)}\defi Im (\delta_{h(y)})=\{f(h(y):f\in\sA\}$$
which is a subgroup of $G$ for all $y\in Y$,
and denote by $\hbox{Hom}(G_{h(y)},G_y)$ the set of all group homomorphisms on $G_{h(y)}$ into $G_y$.
Consider now the set
$$\sG\defi \bigcup\limits_{y\in Y}\hbox{Hom} (G_{h(y)},G_y).$$

We can think of the elements of $\sG$ as partial functions on $G$.
That is, functions $\ga : \hbox{Dom}(\ga)\subseteq G\longrightarrow G$
whose domain is a (not necessarily proper) subset of $G$. Since the group $G$ is discrete,
we can equip $\sG$ with the product (or pointwise convergence) topology as follows:

Let $[\ga;g_1,\dots, g_n]\defi \{\gb\in G^G : \ga(g_i)=\gb(g_i),\ g_i\in G,\ 1\leq i\leq n \}$
be a basic neighborhood of a map $\ga\in G^G$.
If now $\ga$ is a partial map, we can restrict this basic neighborhood to $\sG$ by letting $[\ga;g_1,\dots, g_n]$ be the set of all partial maps
$\gb : \hbox{Dom}(\gb)\subseteq G\longrightarrow G$ such that $g_1,\dots, g_n\in \hbox{Dom}(\gb)$ and $\ga(g_i)=\gb(g_i),\ 1\leq i\leq n$.
It is easily verified that this procedure extends the pointwise convergence topology on $\sG$ (cf. \cite{Ore_Tsa:aim}).

With this notation, we define  $\omega\colon Y\to \sG$ by
$$\omega[y](f(h(y)))\defi Hf(y)$$
for each $y\in Y$. We shall see next that $\omega$ is well defined and continuous.

\begin{prop}\label{prp:peso}
With the terminology established above, the following assertions are true:
\begin{enumerate}
\item $\omega[y]$ is a well defined group homomorphism of $G_{h(y)}$ into $G_y$ for all $y\in Y$.
\item $\omega$ is continuous when $\sG$ is equipped with the pointwise convergence topology.
\end{enumerate}
\end{prop}
\begin{proof}
$(1)$ In order to prove that $\omega[y]$ is well defined, take $f_1,f_2\in\sA$ such that $f_1(h(y))=f_2(h(y))$. By Proposition \ref{propiedadessoporte},
we have $\omega[y](f_1(h(y)))=Hf_1(y)=Hf_2(y)=\omega[y](f_2(h(y)))$. The verification that $\omega[y]$ is a group homomorphism
is easy and it is left to the reader.

$(2)$ Let $(y_d)_{d\in D}$ be a net converging to $y$ in $Y$. If $g$ is an arbitrary element in $\hbox{Dom}(\omega[y])=G_{h(y)}$,
then $g=f(h(y))\in G_{h(y)}$ for some $f\in\sA$. Since $h$ is continuous, $f\circ h\in C(Y,G)$, by Proposition \ref{propiedadesh},
and $(f\circ h)^{-1}(g)$ is a clopen neighborhood of $y\in Y$. Since $G$ is discrete, there is $d_1(g)\in D$ such that
$f(h(y_d))=g$ for all $d\geq d_1(g)$. Thus $g\in \hbox{Dom}(\omega[y_d])$ for all $d\geq d_1(g)$.
In like manner, as $\omega[y](g)=Hf(y)=g_y\in G$ and $Hf\in C(Y,G)$, we have that $(H f)^{-1}(g_y)$ is a clopen neighborhood of $y$.
As a consequence there is $d_2\geq d_1(g)$ such that $Hf(y_d)=g_y$ for all $d\geq d_2$.
Thus $\omega[y_d](g)=Hf(y_d)=g_y=Hf(y)=\omega[y](g)$ for all $d\geq d_2$. This means that
the net $(\omega[y_d])_{d\in D}$ converges to $\omega[y]$ in the pointwise convergence topology over $\sG$.
\end{proof}
\mkp

Observe that, since $G$ is discrete, the compact subsets in $G$ are all finite. Therefore, we have also proved that $\omega$ is also continuous
if we consider the compact open topology on $\sG$.
We are in position now of establishing a main result in this paper.

\bthm\label{th:3.1}
Let $H\colon\sA \to\sB$ be a separating group homomorphism. 
Then there are continuous maps $$h\colon Y\to X$$ and $$\omega\colon Y\to \bigcup\limits_{y\in Y}\emph{Hom} (G_{h(y)},G_y)$$
satisfying the following properties:
\begin{enumerate}
\item For each $y\in Y$ and every $f\in \sA$ it holds $$Hf(y)=\omega[y](f(h(y))).$$
\item $H$ is continuous with respect to the pointwise convergence topology.
\item  $H$ is continuous with respect to the compact open topology.
\item If $H$ is biseparating bijection of $\sA$ onto $\sB$, then $h$ is a homeomorphism.
\end{enumerate}
\ethm
\begin{proof}
Item $(1)$ is consequence of the definition of $\omega$ and (2) follows from assertion (2) in Proposition \ref{prp:peso}.
Thus, only $(3)$ needs to be verified.

 Let $(f_d)_{d\in D}\in\sA$ be a net converging to $e_G$ in the compact open topology. If $K$ is a compact subset of $Y$,
 then $h(K)$ is a compact subset in $X$ by the continuity of $h$. Therefore $(f_d)_d$ is eventually the constant function $e_G$ on $h(K)$.
 Applying (1), it follows that $(Hf_d)_{d\in D}$ is eventually $e_G$ on $K$, which completes the proof.
\end{proof}

\begin{cor}\label{cor:3.2}
Let $H\colon\sA \to\sB$ be a separating group homomorphism, where 
$\sA$ is pointwise dense. 
Then there are continuous maps $h\colon Y\to X$ and $\omega\colon Y\to \E nd (G)$ satisfying the following properties:
\begin{enumerate}
\item For each $y\in Y$ and every $f\in \sA$ it holds $$Hf(y)=\omega[y](f(h(y))).$$
\item $H$ is continuous with respect to the pointwise convergence topology.
\item  $H$ is continuous with respect to the compact open topology.
\item If $H$ is a biseparating bijection of $\sA$ onto $\sB$, then $h$ is a homeomorphism.
\end{enumerate}
\end{cor}
\mkp

We are in now position of establishing the results formulated at the Introduction.
\mkp

\begin{proof}[Proof of Theorem 1.1]\ After Theorem \ref{th:3.1} and Corollary \ref{cor:3.2}, we only need to
verify that $\omega[y]\in \A ut(G)$ for all $y\in Y$. Applying Theorem \ref{th:3.1} to $H^{-1}$, we obtain
maps $$\rho\colon X\to \E nd(G)$$ and $$k\colon X\to Y$$
such that for every $x\in X$ and $g\in \sB$, we have  $$H^{-1}g(x)=\rho[x](g(k(x))).$$

Thus, for every $f\in\sA$ and $x\in X$, we have
$$f(x)=H^{-1}\circ (Hf)(x)=\rho[x](Hf(k(x)))=\rho[x](\omega[k(x)](f(h(k(x)))))$$

\noindent which means that the support subset of $\delta_{x}\circ (H^{-1}\circ H)$ is both $x$ and $h(k(x))$.
Since the support of a separating map is unique, this means that $h\circ k=id_X$, which implies
that $k$ is a right inverse of $h$. Analogously, for every $g\in\sB$ and $y\in Y$, we have
$$g(y)=H\circ (H^{-1}g)(y)=\omega[y](H^{-1}(g(h(y)))=\omega[y](\rho[h(y)](g(k(h(y)))))$$

\noindent which means that the support subset of $\delta_y\circ (H\circ H^{-1})$ is both $y$ and $k(h(y))$.
Again, this implies that $k$ is a left inverse of $h$. Being $k$ both left and right inverse of $h$
yields that $k=h^{-1}$. Therefore

$$f(x)=H^{-1}\circ (Hf)(x)=\rho[x](Hf(k(x)))=\rho[x](\omega[k(x)](f(x)))$$
and
$$g(y)=H\circ (H^{-1}g)(y)=\omega[y](H^{-1}(g(h(y)))=\omega[y](\rho[h(y)](g(y))).$$

Applying the former equality to $x=h(y)$, it follows that $\rho[h(y)]\circ\omega[y]=id_G$ for all $y\in Y$,
and from the latter, we also have that $\omega[y]\circ\rho[h(y)]=id_G$.
This means that $\omega[y]$ has left and right inverse and, therefore, it is an automorphism on $G$,
which completes the proof.
\end{proof}
\mkp

\begin{proof}[Proof of Corollary 1.2]\
It follows directly from Theorem 1.1.
\end{proof}

\end{document}